\documentclass[a4paper,reqno]{amsart}
\usepackage{amssymb}
\usepackage[vcentermath]{youngtab}
\usepackage{pstricks,pst-node,multido}

\def\tab(#1){\,\mbox{\tiny$\young(#1)$}\,}

\title[Large Dimensional homomorphism spaces]{Large dimensional homomorphism spaces between Weyl modules and Specht modules}

\author[S.~Lyle]{Sin\'ead Lyle}
\address{School of Mathematics, University of East Anglia, Norwich NR4 7TJ, UK.}
\email{s.lyle@uea.ac.uk}
\numberwithin{equation}{section}
\numberwithin{figure}{section}
\newtheorem{lemma}{Lemma}[section]
\newtheorem{theorem}[lemma]{Theorem}
\newtheorem{proposition}[lemma]{Proposition}

\newtheorem*{theorem*}{Theorem}
\newtheorem*{proposition*}{Proposition}
\newtheorem*{lemma*}{Lemma}
\newtheorem*{definition*}{Definition}
\newtheorem*{caution*}{Caution}

\theoremstyle{definition}

\theoremstyle{remark}

\newtheorem*{ex}{Example}

\newcommand{\h}{\mathcal{H}}
\newcommand{\sym}{\mathfrak{S}} 
\newcommand{\la}{\lambda}
\newcommand{\gauss}[2]{{{#1} \brack {#2}}}

\DeclareMathOperator{\Hom}{Hom}
\DeclareMathOperator{\EHom}{EHom}
\def\S{\mathsf{S}}
\def\T{\mathsf{T}}
\def\R{\mathsf{R}}
\def\U{\mathsf{U}}
\def\A{\mathsf{A}}
\def\B{\mathsf{B}}
\DeclareMathOperator{\rrr}{r}
\newcommand{\RowT}{\mathcal{T}_{\rrr}}
\newcommand{\rep}[3]{\begin{array}{l}{#1}\\{#2}\\{#3}\end{array}}

\begin{document}
\begin{abstract}
We give a family of pairs of Weyl modules for which the corresponding homomorphism space is at least 2-dimensional.  Using this result we show that for fixed parameters $e>0$ and $p\geq 0$ there exist arbitrarily large homomorphism spaces between pairs of Weyl modules. 
\end{abstract}
\maketitle

\section{Introduction}
Let $F$ be a field of characteristic $p \geq 0$.  Take $q \in F^\times$ with the property that $1+q+\ldots+q^{f-1}=0 \in F$ for some integer $2 \leq f < \infty$ and let $e\geq 2$ be minimal with this property.  For $n \geq 0$, we write $\h_n=\h_{F,q}(\sym_n)$ to denote the Hecke algebra of the symmetric group $\sym_n$ and $\mathcal{S}_{n}=\mathcal{S}_{F,q}(\sym_n)$ to denote the corresponding $q$-Schur algebra.  For each partition $\mu$ of $n$, we may define a $\h_n$-module $S^\mu$, known as a Specht module, and an $\mathcal{S}_n$-module $\Delta(\mu)$, known as a Weyl module.  Recall that
if $\mu$ and $\la$ are partitions of $n$ then
\[\dim(\Hom_{\h_n}(S^\mu,S^\la)) \geq \dim(\Hom_{\mathcal{S}_n}(\Delta(\mu),\Delta(\la))\]
with equality if $q \neq -1$~\cite{DJ}.  Despite much investigation, there are few known examples of Weyl modules $\Delta(\mu)$ and $\Delta(\la)$ such that 
$\dim(\Hom_{\mathcal{S}_n}(\Delta(\mu),\Delta(\la)))>1$.  The first such pairs were recently exhibited by Dodge~\cite{Dodge}.  Working in the symmetric group algebra and using results of Chuang and Tan~\cite{ChuangTan} on the radical filtrations of Specht modules belonging to Rouquier blocks, he showed that for any $k$ satisfying $k(k+1)/2 +1 <p$ there exist partitions $\mu$ and $\la$ of some integer $n$ such that $\dim(\Hom_{F\sym_n}(S^\mu,S^\la))=k$.  In particular, for $p \geq 5$ there exist Specht modules, and hence Weyl modules, such that the corresponding homomorphism space is at least 2-dimensional. Using Lemma~\ref{DimGrows} below, Dodge's result proves the following: Let $F$ be a field of characteristic $p \geq 5$.  Then given any integer $l \geq 0$ there exist partitions $\alpha$ and $\beta$ of some integer $m$ such that $\dim(\Hom_{F\sym_m}(S^\alpha,S^\beta))\geq l$.  

\begin{lemma} \label{DimGrows}
Suppose $\mu$ and $\lambda$ are partitions of an integer $n$ such that $\dim(\Hom_{\mathcal{S}_n}(\Delta(\mu),\Delta(\la)))=k$.  Then there exist partitions $\alpha$ and $\beta$ of some integer $m$ such that $\dim(\Hom_{\mathcal{S}_m}(\Delta(\alpha),\Delta(\beta)))=k^2$. 
\end{lemma}

\begin{proof}
We may assume $k \geq 1$.  If $\mu=(\mu_1,\ldots,\mu_a)$ and $\la=(\la_1,\ldots,\la_b)$ then, since $\Hom_{\mathcal{S}_n}(\Delta(\mu),\Delta(\la)) \neq \{0\}$, we have $\la \unrhd \mu$ so that $a \geq b$ and $\la_1 \geq \mu_1$.  Define partitions $\alpha$ and $\beta$ by
\begin{align*}
\alpha_i & = \begin{cases}
\mu_i + \la_1, & 1 \leq i \leq a, \\
\mu_{i-a}, & a+1 \leq i \leq 2a,
\end{cases} &
\beta_i & = \begin{cases}
\la_i + \la_1, & 1 \leq i \leq a, \\
\la_{i-b}, & a+1 \leq i \leq 2a,
\end{cases} 
\end{align*}
so that

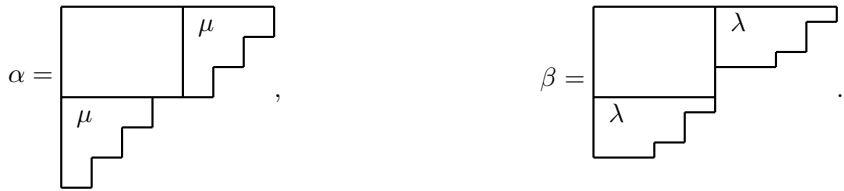
\pspicture(-2,-.5)(30,2.75)
\psset{unit=0.2cm}
\put(-3.5,7){$\alpha=$}
\put(9,10.4){$\mu$}
\put(1,4.4){$\mu$}
\psline(0,0)(0,12)
\psline(0,12)(14,12)
\psline(14,12)(14,10)
\psline(14,10)(12,10)
\psline(12,10)(12,8)
\psline(12,8)(10,8)
\psline(10,8)(10,6)
\psline(10,6)(0,6)
\psline(6,6)(6,4)
\psline(6,4)(4,4)
\psline(4,4)(4,2)
\psline(4,2)(2,2)
\psline(2,2)(2,0)
\psline(2,0)(0,0)
\psline(8,12)(8,6)
\put(14,6){,}
\put(31.5,6.75){$\beta=$}
\put(44,10.4){$\la$}
\put(36,4.4){$\la$}
\psline(35,2)(35,12)
\psline(35,12)(51,12)
\psline(51,12)(51,11)
\psline(51,11)(49,11)
\psline(49,11)(49,9)
\psline(49,9)(47,9)
\psline(47,9)(47,8)
\psline(47,8)(43,8)
\psline(43,12)(43,5)
\psline(43,5)(41,5)
\psline(41,5)(41,3)
\psline(41,3)(39,3)
\psline(39,3)(39,2)
\psline(39,2)(35,2)
\psline(35,6)(43,6)
\put(51,6){.}
\endpspicture 

Then $\dim(\Hom_{\mathcal{S}_n}(\Delta(\alpha),\Delta(\beta)))=k^2$ by the generalized row and column removal theorems~\cite[Theorem 3.1]{LM:rowhoms} or~\cite[Prop.~10.4]{Donkin:tilting}.   
\end{proof}

In this paper, we exhibit pairs of partitions such that the homomorphism space between the corresponding Weyl modules is at least 2-dimensional.  In fact, we believe that it is exactly 2-dimensional, but this would be considerably harder to prove.  

\begin{theorem}  \label{Main}
For $a \geq b \geq c+1 \geq 4$, define partitions 
\begin{align*}
\mu&=\mu(a,b,c,e) = (ae-3,be-3,ce-3,e-1,e-1), \\
\la&=\la(a,b,c,e) = ((a+2)e-5,be-3,ce-3)),
\end{align*}
of some integer $n$.  
Then $\dim(\Hom_{\mathcal{S}_n}(\Delta(\mu),\Delta(\la))) \geq 2$.  
\end{theorem}

Using Lemma~\ref{DimGrows}, this is sufficient to prove the following result.

\begin{theorem} \label{ArbLarge}
Given any integer $l \geq 0$ there exist partitions $\alpha$ and $\beta$ of some integer $m$ such that 
\begin{align*}
\dim(\Hom_{\mathcal{S}_{m}}(\Delta(\alpha),\Delta(\beta)))&\geq l; \text{ and  hence}\\
\dim(\Hom_{\h_m}(S^\alpha,S^\beta))&\geq l.
\end{align*}
\end{theorem}

If the results of Chaung and Tan~\cite{ChuangTan} hold for the $q$-Schur algebra (rather than just the Schur algebra) then the proof of Theorem~\ref{ArbLarge} almost follows from the work of Dodge (and Lemma~\ref{DimGrows}): only the cases $e=2,3,4$ would not be covered.  We note that Lemma~\ref{DimGrows} is the only result we know that allows us to build large homomorphism spaces from smaller ones; for example, for small $e$ we do not know of any pair of partitions such that the homomorphism space between the corresponding Weyl modules has dimension 3.      

\section{Proof of Theorem~\ref{Main}}
In this section, we give the proof of the main result.  Fix a field $F$ and an element $q \in F^\times$ such that $e=\min\{f\geq 2 \mid 1+q+\ldots+q^{f-1}=0\}$ exists. For $n \geq 0$ let $\mathcal{S}_n = \mathcal{S}_{F,q}(\sym_n)$ and $\h_n=\h_{F,q}(\sym_n)$.  The characteristic of the field plays no further role in this paper.  We first recall a method to determine the dimension of the homomorphism space between a pair of Weyl modules.   For full details, we refer the reader to~\cite[Section 2.2]{L:Construct}.

\subsection{Homomorphism spaces}
Fix partitions $\la$ and $\mu$ of $n$.  For every composition $\nu$ of $n$, we define $m_\nu \in \h_n$ and a cyclic right $\h_n$-module $M^\nu=m_\nu \h$.  
Let $\mathcal{T}_{\text{r}}(\la,\nu)$ denote the set of row-standard $\la$-tableaux of type $\nu$, with $\mathcal{T}_0(\la,\nu)\subseteq \mathcal{T}_{\text{r}}(\la,\nu)$ the subset of semistandard tableaux.  For each $\T \in \mathcal{T}_{\text{r}}(\la,\nu)$ we define a $\h_n$-homomorphism $\Theta_\T:M^\nu \rightarrow S^\la$ such that $\{\Theta_\T \mid \T \in \mathcal{T}_0(\la,\nu)\}$ are linearly independent.

Let $\ell(\nu)$ denote the number of parts of any composition $\nu$.  
For $1 \leq d < \ell(\mu)$ and $1 \leq t \leq \mu_{d+1}$ we define an element $h_{d,t} \in \h_n$.  Let $\EHom_{\h_n}(M^\mu,S^\la)$ be the space spanned by $\{\Theta_\T \mid \T \in \mathcal{T}_0(\la,\mu)\}$ and let
\[\Psi(\mu,\la) = \{\Theta \in \EHom_{\h_n}(M^\mu,S^\la) \mid \Theta(m_\mu h_{d,t}) =0 \text{ for all } 1 \leq d < \ell(\mu), 1 \leq t \leq \mu_{d+1} \}.\]
This definition is motivated by the following result which follows from~\cite[Theorem 2.2]{L:Construct} and the remark following~\cite[Corollary 2.4]{L:Construct}.

\begin{lemma} \label{EqualSpace}
\[\Psi(\mu,\la) \cong_F \Hom_{\mathcal{S}_n}(\Delta(\mu),\Delta(\la)).\]
\end{lemma}

We therefore want to determine $\Psi(\mu,\la)$.  First we set up some notation. If $\T \in \mathcal{T}_{\text{r}}(\la,\nu)$, let $\T^i_j$ denote the number of entries of $\T$ which lie in row $j$ and which are equal to $i$.  We extend this definition by setting $\T^{>i}_j = \sum_{k>i} \T^k_j$, and similarly for other definitions.  If $m \geq 0$ define 
\[ [m] = 1+q+\ldots + q^{m-1} \in F.\]
Let $[0]!=1$ and for $m \geq 1$, set $[m]! = [m][m-1]!$.  If $m \geq j \geq 0$, set 
\[\gauss{m}{j} = \frac{[m]!}{[j]![m-j]!}.\]
For integers $m$ and $j$, if any of the conditions $m \geq j \geq 0$ fail we define $\gauss{m}{j}=0$.  Using Lemma~\ref{GaussSum} below or otherwise, it is straightforward to show that $\gauss{m}{k}$ is then well-defined for any $m,k \in \mathbb{Z}$, and may be considered as an element of $\mathbb{Z}[q]$.    

\begin{lemma}[\cite{L:Construct} Proposition 2.7] \label{Lemma5}  
Suppose that $\T \in \RowT(\la,\mu)$.   
Choose $d$ with $1 \leq d <\ell(\mu)$ and $t$ with $1 \leq t \leq \mu_{d+1}$.  Let $\mathcal{S}$ be the set of row-standard tableaux obtained by replacing $t$ of the entries in $\T$ which are equal to $d+1$ with $d$.  Each tableau $\S \in \mathcal{S}$ will be of type $\nu(d,t)$ where
\[\nu(d,t)_j = 
\begin{cases}
\mu_j+t, & j=d, \\
\mu_j-t, & j=d+1, \\
\mu_j, & \text{otherwise}.
\end{cases} \\ \]  
Recall that $\Theta_\T: M^\mu \rightarrow S^\la$ and $\Theta_{\S}:M^{\nu(d,t)} \rightarrow S^\la$.  Then 
\[\Theta_\T(m_\mu h_{d,t}) = \sum_{\S \in \mathcal{S}} \left(  \prod_{j= 1}^{\ell(\la)} q^{\T^d_{>j}(\S^d_j - \T^d_j)} \gauss{\S^d_j}{\T^d_j}\right) \Theta_\S (m_{\nu(d,t)}).\] 
\end{lemma}

\begin{lemma}[\cite{L:Construct} Proposition 2.9]  \label{Lemma7}
Suppose $\la$ is a partition of $n$ and $\nu$ is a composition of $n$.
Let $\S \in \RowT(\la,\nu)$.  
Suppose $1 \leq r\leq \ell(\la)-1$ and that $1 \leq d \leq \ell(\nu)$.  Let
\[\mathcal{G} =\left\{g=(g_1,g_2,\ldots,g_{\ell(\nu)}) \mid g_d=0, \, \sum_{i=1}^{\ell(\nu)} g_i =\S^d_{r+1} \text{ and } g_i \leq \S^{i}_{r} \text{ for } 1 \leq i \leq \ell(\nu)\right\}.\]
For $g \in \mathcal{G}$, let $\bar{g}_{d-1} = \sum_{i=1}^{d-1}g_i$ and let $\U_g$ be the row-standard tableau formed from $\S$ by moving all entries equal to $d$ from row $r+1$ to row $r$ and for $i \neq d$ moving $g_i$ entries equal to $i$ from row $r$ to row $r+1$. Then
\[\Theta_\S = (-1)^{\S^d_{r+1}} q^{-\binom{\S^d_{r+1}+1}{2}} q^{-\S^d_{r+1}S^{<d}_{r+1}} \sum_{g \in \mathcal{G}} q^{\bar{g}_{d-1}} \prod_{i=1}^{\ell(\nu)} q^{g_i \S^{<i}_{r+1}} \gauss{\S^i_{r+1}+g_i}{g_i}\Theta_{\U_g}.\]
\end{lemma}

In the following section, we apply these two lemmas to find elements of $\Psi(\mu,\la)$.  

\begin{ex} \label{Ex1} Let $e=2$.  Take $\la=(7,5,3)$ and $\mu=(5,5,3,1,1)$.  We identify a $\la$-tableau $\T$ of type $\nu \unrhd \mu$ with the image $\Theta_\T(m_{\nu}) \in S^\la$.  Recall that if $\la \not \! \unrhd \nu$ then $\mathcal{T}_0(\la,\nu) = \emptyset$ so that we immediately have $\Theta(m_\mu h_{1,t}) = 0$ for $t=3,4,5$ and $\Theta(m_\mu h_{2,3})= 0$.  
\begin{enumerate}
\item Let $\Theta(m_\mu)=\tab(1111123,22223,345)$.  Then
\begin{align*}
\Theta(m_\mu h_{4,1}) & = [2] \tab(1111123,22223,344) \\ &= 0, \\
\Theta(m_\mu h_{3,1}) & = [2] \tab(1111123,22223,335)\\ & = 0, \\
\Theta(m_\mu h_{2,1}) & = q^4 [2] \tab(1111122,22223,345)  +  [5] \tab(1111123,22222,345) + \tab(1111123,22223,245)  \\
&=  q^4 [2] \tab(1111122,22223,345)  +  [5] \tab(1111123,22222,345) - \tab(1111123,22222,345)\\& =0, \\
\Theta(m_\mu h_{2,2}) & = q^4 [2][5] \tab(1111122,22222,345)  +  q^4 [2] \tab(1111122,22223,245) + [5] \tab(1111123,22222,245)  \\
&=  q^4 [2][5] \tab(1111122,22222,345)  -q^4 [2] \tab(1111122,22222,345) \\& =0, \\
\end{align*}
\begin{align*}
\Theta(m_\mu h_{1,1}) & = [6] \tab(1111113,22223,345) + \tab(1111123,12223,345) \\
 & = [6] \tab(1111113,22223,345) - [4] \tab(1111113,22223,345) - q^3 [2] \tab(1111112,22233,345)\\ &=0, \\
\Theta(m_\mu h_{1,2}) & = [6] \tab(1111113,12223,345) + \tab(1111123,11223,345) \\
 & = -q^3[6][2]\tab(1111111,22233,345) + q^3[3][2]\tab(1111111,22233,345)\\ &=0, \\
\end{align*}    
so that $\Theta \in \Psi(\mu,\la)$.  
\item Let 
\[\Phi=\tab(1111125,22224,333) + \tab(1111124,22225,333) + \tab(1111125,22223,334)+\tab(1111124,22223,335) + \tab(1111123,22225,334) + \tab(1111123,22224,335).\]
Then $\Phi \in \Psi(\mu,\la)$.  
\end{enumerate}
\end{ex}

\subsection{Gaussian Polynomials}
In order to tell if a homomorphism $\Theta$ lies in $\Psi(\mu,\la)$ we record some results about the Gaussian polynomials $\gauss{m}{j}$.  The first is well-known. 

\begin{lemma} \label{GaussSum}
Suppose $m,j \geq 0$.  Then
\begin{align*}
\gauss{m+1}{j} & = \gauss{m}{j-1} + q^{j} \gauss{m}{j} \\
&=\gauss{m}{j} + q^{m-j+1}\gauss{m}{j-1}. 
\end{align*}
\end{lemma}

\begin{lemma} [\cite{L:Construct} Lemma~2.6] \label{GaussLemma}
Suppose $m, k \geq l \geq 0$.  Then,
\[\sum_{j \geq 0}(-1)^j q^{\binom{j}{2}}\gauss{l}{j}\gauss{m-j}{k} = q^{l(m-k)}\gauss{m-l}{k-l}.\]
\end{lemma}

\begin{lemma} \label{GaussDie1}
Suppose that $m \geq 0$ and write $m=m^\ast e+m'$ where $0 \leq m' <e$.  If $m'<j \leq e-1$ then \[\gauss{m}{j}=0.\]  
\end{lemma}

\begin{proof}
Write
\[\gauss{m}{j} = \frac{[m][m-1]\ldots [m-j+1]}{[j][j-1]\ldots [1]}\]
so that one of the terms in the numerator and none of the terms in the denominator are equal to zero.    
\end{proof}

The next lemma follows immediately. 
\begin{lemma} \label{GaussDie}
Suppose $1 \leq j \leq e-1$.  Then 
\[\gauss{ae-1+j}{j} =0\]
for all $a \geq 0$.  
\end{lemma}

\begin{lemma} \label{ThreeProd}
Suppose $m \geq l \geq0$, that $k \geq 1$ and that $a_1,\ldots,a_k \geq 0$ are such that $\sum_{i=1}^k a_i=m$.  Then
\[\sum_{c_1+\ldots+c_k=l} \prod_{i=1}^k q^{(a_i-c_i) (c_{i+1}+\ldots+c_k)} \gauss{a_i}{c_i} = \gauss{m}{l}.\]
\end{lemma}

\begin{proof}
The result is true for $m=0$ so suppose that $m \geq 1$ and that the lemma holds for $m-1$.  Using Lemma~\ref{GaussSum} and the inductive hypothesis,
\begin{align*}
\sum_{c_1+\ldots+c_k=l}  \prod_{i=1}^k & q^{(a_i-c_i) (c_{i+1}+\ldots+c_k)} \gauss{a_i}{c_i} \\
& = \sum_{c_1+\ldots+c_k=l}  \left( \prod_{i=1}^{k-1} q^{(a_i-c_i) (c_{i+1}+\ldots+c_k)} \gauss{a_i}{c_i} \right) \left(\gauss{a_k-1}{c_k}+q^{a_k-c_k}\gauss{a_k-1}{c_k-1}\right)  \\
&=\sum_{c_1+\ldots+c_k=l}  \left(\prod_{i=1}^{k-1} q^{(a_i-c_i) (c_{i+1}+\ldots+c_k)} \gauss{a_i}{c_i}\right) \gauss{a_k-1}{c_k} \\
&\hspace*{0.75cm} + q^{a_1+\ldots+a_k-c_1-\ldots-c_k} \sum_{c_1+\ldots+c_k=l-1}  \left(\prod_{i=1}^{k-1} q^{(a_i-c_i) (c_{i+1}+\ldots+c_k)} \gauss{a_i}{c_i}\right) \gauss{a_k-1}{c_k} \\
&=\gauss{m-1}{l}+q^{m-l}\gauss{m-1}{l-1} \\
&=\gauss{m}{l}
\end{align*}  
as required. 

An alternative proof may be constructed by counting the number of $l$-dimensional vector spaces of an $m$-dimensional vector space over the finite fields.  
\end{proof}

\subsection{Elements of $\Psi(\mu,\la)$}
We are now ready to prove Theorem~\ref{Main}.  Fix $a \geq b \geq c+1 \geq 4$ and define partitions 
\begin{align*}
\mu&=\mu(a,b,c,e) = (ae-3,be-3,ce-3,e-1,e-1), \\
\la&=\la(a,b,c,e) = ((a+2)e-5,be-3,ce-3)),
\end{align*}
of some integer $n$.  
If $\T \in \mathcal{T}_{\text{r}}(\la,\nu)$ for some $\nu \unrhd \mu$, recall that $\T^i_j$ is the number of entries equal to $i$ in row $j$ of $\T$.  We denote $\T$ by  
\[\T=\rep{1^{\T^1_1} \, 2^{\T^2_1}\,3^{\T^3_1}\, 4^{\T^4_1}\,5^{\T^5_1}}{1^{\T^1_2} \, 2^{\T^2_2}\,3^{\T^3_2}\, 4^{\T^4_2}\,5^{\T^5_2} }{1^{\T^1_3} \, 2^{\T^2_3}\,3^{\T^3_3}\, 4^{\T^4_3}\,5^{\T^5_3}},\]
where we omit terms if $\T^i_j=0$.  
Our strategy is to define linearly independent elements $\Theta$ and $\Phi$ in $\EHom_{\h_n}(S^\mu,S^\la)$ and use Lemmas~\ref{Lemma5} and Lemma~\ref{Lemma7} to show that $\Theta(m_\mu h_{d,t}) = \Phi(m_\mu h_{d,t})=0$ for all $1 \leq d \leq 4$ and $1 \leq t \leq \mu_{d+1}$.  Theorem~\ref{Main} then follows by Lemma~\ref{EqualSpace}. 

\begin{lemma} \label{AllTogetherNow}
Suppose that $\T \in \mathcal{T}_0(\la,\mu)$ has the form
\[\T = \rep{1^{ae-3} \, 2^{e-1}\,3^{\T^3_1}\,4^{\T^4_1}\,5^{\T^5_1}}{2^{(b-1)e-2}\,3^{\T^3_2}\,4^{\T^4_2}\,5^{\T^5_2}}{3^{\T^3_3}\,4^{\T^4_3}\,5^{\T^5_3}}.\]
Then the following results hold.  
\begin{enumerate}
\item Suppose $1 \leq t \leq e-1$.  Write $\T \xrightarrow{4,t}\S$ if $\S$ is a row-standard $\nu$-tableau formed from $\T$ by changing $t$ entries equal to 5 in $\T$ into 4s.  If $\T \xrightarrow{4,t}\S$ then $\S$ is semistandard and
\[\Theta_\T(m_\mu h_{4,t}) = \sum_{\T \xrightarrow{4,t}\S} q^{(\T^4_3+\T^4_2)(\S^4_1-\T^4_1)} \gauss{\S^4_1}{\T^4_1} q^{\T^4_3(\S^4_2-\T^4_2)}\gauss{\S^4_2}{\T^4_2}\gauss{\S^4_3}{\T^4_3} \Theta_\S(m_\nu).\]
\item Suppose $1 \leq t \leq e-1$.  Write $\T \xrightarrow{3,t}\S$ if $\S$ is a row-standard $\nu$-tableau formed from $\T$ by changing $t$ entries equal to 4 in $\T$ into 3s.  If $\T \xrightarrow{3,t}\S$ then $\S$ is semistandard and
\[\Theta_\T(m_\mu h_{3,t}) = \sum_{\T \xrightarrow{3,t}\S} q^{(\T^3_2+\T^3_3)(\S^3_1-\T^3_1)} \gauss{\S^3_1}{\T^3_1} q^{\T^3_3(\S^3_2-\T^3_2)}\gauss{\S^3_2}{\T^3_2}\gauss{\S^3_3}{\T^3_3} \Theta_\S(m_\nu).\]
\item Suppose $1 \leq t \leq \mu_3-1$.  Write $\T \xrightarrow{2,t}\S$ if $\S$ is a row-standard $\nu$-tableau formed from $\T$ by first changing $t$ entries equal to 3 in $\T$ into 2s in the second and third rows and then exchanging all entries equal to $2$ in row $3$ with entries not equal to $2$ in row $2$.  
If $\T \xrightarrow{2,t}\S$ then $\S$ is semistandard and
\[\Theta_\T(m_\mu h_{2,t}) = \sum_{\T \xrightarrow{2,t}\S} (-1)^{\T^3_3-\S^3_3} q^{\binom{\T^3_3-\S^3_3}{2}+\S^3_3 t} \gauss{(b-1)e-2+t-\T^3_3}{(b-1)e-2-\S^3_3}  q^{\T^4_3(\S^5_3-\T^5_3)}\gauss{\S^4_3}{\T^4_3} \gauss{\S^5_3}{\T^5_3}\Theta_\S(m_\nu) .\]
In particular, $\Theta_\T(m_\mu h_{2,t})=0$ for $t>e-1$.  
\item Suppose $1 \leq t \leq \mu_2-1$.  Write $\T \xrightarrow{1,t}\S$ if $\S$ is a row-standard $\nu$-tableau formed from $\T$ by first changing $t$ entries equal to 2 in $\T$ into 1s and then exchanging all entries equal to $1$ in row $2$ with entries not equal to $1$ in row $1$.  
If $\T \xrightarrow{1,t}\S$ then $\S$ is semistandard and
\begin{multline*} \Theta_\T(m_\mu h_{1,t}) = \sum_{\T \xrightarrow{1,t}\S} (-1)^{\T^2_2-\S^2_2} q^{\binom{\T^2_2-\S^2_2}{2}+\S^2_2 t}\gauss{(a-b+1)e-1+t}{ae-3-\S^2_2}\\ q^{\T^3_2(\S^4_2-\T^4_2)}q^{(\T^3_2+\T^4_2)(\S^5_2-\T^5_2)} \gauss{\S^3_2}{\T^3_2}\gauss{\S^4_2}{\T^4_2}\gauss{\S^5_2}{\T^5_2}\Theta_\S(m_\nu) \end{multline*}
where $\T^2_2 = (b-1)e-2$.  In particular, $\Theta_\T(m_\mu h_{1,t})=0$ for $t>2e-2$.   
\end{enumerate}
\end{lemma}

\begin{proof}
To check the tableaux $\S$ are semistandard, observe that $ae-3 \geq be-3$ and that $(b-1)e-2 \geq ce-3$.  Parts (1) and (2) are then just restatements of Lemma~\ref{Lemma5}.  Now consider (3).  Use Lemma~\ref{Lemma5} to write $\Theta_\T$ as a linear combination of terms $\Theta_\R(m_\nu)$ where $\R$ is formed from $\T$ by changing entries equal to 3 into 2s.  If $s>0$ entries are changed in the first row then the term occurs with coefficient a multiple of  $\gauss{e-1+s}{s}=0$ by Lemma~\ref{GaussDie} so we may assume all entries changed are in the last two rows.  It then follows from Lemma~\ref{Lemma7} that  $\Theta_\T(m_\mu h_{2,t}) = \sum_{\T \xrightarrow{2,t}\S} b(\S) \Theta_\S(m_\nu)$ where
\[b(\S) = \sum_{j \geq 0} (-1)^j q^{-\binom{j+1}{2}} q^{j(j-\T^3_3+\S^3_3)} \gauss{(b-1)e-2+t-j}{t-j} q^{\T^3_3(\S^4_3-\T^4_3)}q^{(\T^3_3+\T^4_4)(\S^5_3-\T^5_3)} \gauss{S^3_3}{T^3_3-j}\gauss{S^4_3}{\T^4_3} \gauss{\S^5_3}{\T^5_3}. \ \]   
Changing the limits of the sum and applying Lemma~\ref{GaussLemma} we obtain
\begin{align*}
b(\S) &= (-1)^{\T^3_3-\S^3_3} q^{\T^3_3(\S^4_3-\T^4_3)}q^{(\T^3_3+\T^4_4)(\S^5_3-\T^5_3)} q^{-\binom{\T^3_3-\S^3_3+1}{2}}  \gauss{S^4_3}{\T^4_3} \gauss{\S^5_3}{\T^5_3} \\
& \hspace*{0.75cm} \sum_{j \geq 0} (-1)^j q^{\binom{j}{2}} \gauss{\S^3_3}{j} \gauss{(b-1)e-2+t-j-\T^3_3+\S^3_3}{(b-1)e-2} \\
&=  (-1)^{\T^3_3-\S^3_3}  q^{\T^3_3(\S^4_3-\T^4_3)}q^{(\T^3_3+\T^4_4)(\S^5_3-\T^5_3)} q^{-\binom{\T^3_3-\S^3_3+1}{2}} \gauss{S^4_3}{\T^4_3} \gauss{\S^5_3}{\T^5_3}  q^{\S^3_3(t-\T^3_3+\S^3_3)} \gauss{(b-1)e-2+t-\T^3_3}{(b-1)e-2-\S^3_3} \\
&=  (-1)^{\T^3_3-\S^3_3} q^{\binom{\T^3_3-\S^3_3}{2}+\S^3_3 t} \gauss{(b-1)e-2+t-\T^3_3}{(b-1)e-2-\S^3_3}  q^{\T^4_3(\S^5_3-\T^5_3)}\gauss{\S^4_3}{\T^4_3} \gauss{\S^5_3}{\T^5_3}
\end{align*}
as required.

The proof of part (4) of the lemma follows on identical lines.   
\end{proof}

\begin{proposition} Define a tableau $\T \in \mathcal{T}_0(\la,\mu)$ by
\[\T = \rep{1^{ae-3} \, 2^{e-1}\,3^{e-1}}{2^{(b-1)e-2}\,3^{e-1}}{3^{(c-2)e-1}\,4^{e-1}\,5^{e-1}}\]
and let $\Theta=\Theta_\T$.  Then $\Theta \in \Psi(\mu,\la)$.  
\end{proposition}

\begin{proof}  Note that $\T$ has the form described in Lemma~\ref{AllTogetherNow}.  Suppose $1 \leq t \leq e-1$.  Then applying Lemma~\ref{AllTogetherNow} and Lemma~\ref{GaussDie},
\begin{align*}
\Theta(m_\mu h_{4,t}) & = \gauss{e-1+t}{t} \rep{1^{ae-3} \, 2^{e-1}\,3^{e-1}}{2^{(b-1)e-2}\,3^{e-1}}{3^{(c-2)e-1}\,4^{e-1+t}\,5^{e-1-t}} =0; \\
\Theta(m_\mu h_{3,t}) & = \gauss{(c-2)e-1+t}{t} \rep{1^{ae-3} \, 2^{e-1}\,3^{e-1}}{2^{(b-1)e-2}\,3^{e-1}}{3^{(c-2)e-1+t}\,4^{e-1-t}\,5^{e-1}} =0; \\
\Theta(m_\mu h_{2,t}) 
&= q^{((c-2)e-1)t} \gauss{(b-c+1)e-1+t}{t}\rep{1^{ae-3} \, 2^{e-1}\,3^{e-1}}{2^{(b-1)e-2+t}\,3^{e-1-t}}{3^{(c-2)e-1}\,4^{e-1}\,5^{e-1}} =0.\\
\intertext{Now suppose $1 \leq t \leq 2e-2$.  Then}
\Theta(m_\mu h_{1,t}) &= \sum_{\T \xrightarrow{1,t} \S} (-1)^{\T^2_2-\S^2_2} q^{\binom{\T^2_2-\S^2_2}{2}+\S^2_2t} \gauss{(a-b+1)e-1+t}{ae-3-\S^2_2} \gauss{\S^3_2}{e-1}.
\end{align*}
But if $\T \xrightarrow{d,t} \S$ then $\gauss{\S^3_2}{e-1}=0$ unless $\S^3_2=e-1$; and if $\S^3_2=e-1$ then $1 \leq t \leq e-1$ and then
\[ \gauss{(a-b+1)e-1+t}{ae-3-\S^2_2} = \gauss{(a-b+1)e-1+t}{t}=0\] 
by Lemma~\ref{GaussDie}.  Hence $\Theta(m_\mu h_{d,t})=0$ for all $1 \leq d \leq 4$ and all $1 \leq t \leq \mu_{d+1}$ as required.  
\end{proof}

\begin{proposition}
Let $\mathcal{A}$ denote the set of $\la$-tableaux $\A$ of type $\mu$ which have the form
\[\rep{1^{ae-3} \, 2^{e-1}\,3^{\A^3_1}\, 4^{\A^4_1}\,5^{\A^5_1}}{2^{(b-1)e-2}\,3^{\A^3_2}\, 4^{\A^4_2}\,5^{\A^5_2}}{3^{(c-1)e-2}\,4^{\A^4_3}\,5^{\A^5_3}}\]
and  $\mathcal{B}$ denote the set of $\la$-tableaux $\B$ of type $\mu$ which have the form
\[\rep{1^{ae-3} \, 2^{e-1}\,3^{\B^3_1}\, 4^{\B^4_1}\,5^{\B^5_1}}{2^{(b-1)e-2}\,3^{\B^3_2}\, 4^{\B^4_2}\,5^{\B^5_2}}{3^{(c-1)e-1}\,4^{\B^4_3}\,5^{\B^5_3}}\]
so that all $\A \in \mathcal{A} \cup \mathcal{B}$ are semistandard.  
Set \[\Phi = \sum_{\A \in \mathcal{A}} \Theta_\A - q \sum_{\B \in \mathcal{B}} \Theta_\B.\]
Then $\Phi \in \Psi(\mu,\la)$.  
\end{proposition}

\begin{proof} Note that all tableaux $\A \in  \mathcal{A} \cup \mathcal{B}$ have the form described in Lemma~\ref{AllTogetherNow} and use the notation of that lemma.  For $1 \leq d \leq 4$ and $1 \leq t \leq \mu_{d+1}$, let 
\[\mathcal{D}(d,t)=\{\S \in \mathcal{T}_{0}(\la,\nu) \mid \A \xrightarrow{d,t} \S \text{ for some } \A \in \mathcal{A} \cup \mathcal{B}\}.\] 
For $\S \in \mathcal{D}(d,t)$ define $b_{\mathcal{A}}(\S)$ to be the coefficient of $\Theta_\S(m_\nu)$ in $\sum_{\A \in \mathcal{A}} \Theta_\A(m_\mu h_{d,t})$, define $b_{\mathcal{B}}(\S)$ to be its coefficient in $\sum_{\B \in \mathcal{A}} \Theta_\A(m_\mu h_{d,t})$ and set $b(\S) =  b_{\mathcal{A}}(\S) -q b_{\mathcal{B}}(\S)$ to be its coefficient in $\Phi(m_\mu h_{d,t})$.
 
Take $d=4$ and $1 \leq t \leq e-1$ and suppose that $\S \in \mathcal{D}(4,t)$.  Using Lemma~\ref{AllTogetherNow} and applying Lemma~\ref{ThreeProd} and Lemma~\ref{GaussDie} we have
\begin{align*}
 b_\mathcal{A}(\S)  & = \sum_{\stackrel{\A \in \mathcal{A}}{\A \xrightarrow{d,t} \S}} q^{(\A^4_3+\A^4_2)(\S^4_1-\A^4_1)} \gauss{\S^4_1}{\A^4_1} q^{\A^4_3(\S^4_2-\A^4_2)}\gauss{\S^4_2}{\A^4_2}\gauss{\S^4_3}{\A^4_3} \\
& = \sum_{\A^4_1+\A^4_2+\A^4_3=e-1} q^{(\A^4_3+\A^4_2)(\S^4_1-\A^4_1)} \gauss{\S^4_1}{\A^4_1} q^{\A^4_3(\S^4_2-\A^4_2)}\gauss{\S^4_2}{\A^4_2}\gauss{\S^4_3}{\A^4_3} \\
& = \gauss{\S^4_1+\S^4_2+\S^4_3}{\A^4_1+\A^4_2+\A^4_3}  \\& = \gauss{e-1+t}{t}\\&=0.
\end{align*}
An identical argument shows that $ b_\mathcal{B}(\S)$ is also zero.

Now take $d=3$ and $1 \leq t \leq e-1$.  Suppose that $\S \in \mathcal{D}(d,t)$.  Then
\begin{align*}
b(\S) & = \sum_{\stackrel{\A \in \mathcal{A}}{\A \xrightarrow{d,t} \S}} q^{(\A^3_2+\A^3_3)(\S^3_1-\A^3_1)} q^{A^3_3(\S^3_2-A^3_2)} \gauss{\S^3_1}{\A^3_1} \gauss{\S^3_2}{\A^3_2} \gauss{\S^3_3}{\A^3_3} \\
&\hspace*{0.75cm} -q  \sum_{\stackrel{\B \in \mathcal{B}}{\B \xrightarrow{d,t} \S}} q^{(\B^3_2+\B^3_3)(\S^3_1-\B^3_1)} q^{B^3_3(\S^3_2-B^3_2)} \gauss{\S^3_1}{\B^3_1} \gauss{\S^3_2}{\B^3_2} \gauss{\S^3_3}{\B^3_3} \\
&=q^{((c-1)e-2)((c-1)e-2-\S^3_3+t)}\gauss{S^3_3}{(c-1)e-2} \sum_{A^3_1+\A^3_2=e-1}q^{\A^3_2(\S^3_1-\A^3_1)} \gauss{\S^3_2}{\A^3_2} \gauss{\S^3_1}{\A^3_1}  \\ 
&\hspace*{0.75cm} - q^{((c-1)e-1)((c-1)e-1-\S^3_3+t) +1}\gauss{S^3_3}{(c-1)e-1} \sum_{B^3_1+\B^3_2=e-2}q^{\B^3_2(\S^3_1-\B^3_1)} \gauss{\S^3_2}{\B^3_2} \gauss{\S^3_1}{\B^3_1}  \\ 
&=q^{((c-1)e-2)((c-1)e-2-\S^3_3+t)}\gauss{S^3_3}{(c-1)e-2} \gauss{\S^3_2+\S^3_1}{e-1} -  q^{((c-1)e-1)((c-1)e-1-\S^3_3+t) +1}\gauss{S^3_3}{(c-1)e-1}  \gauss{\S^3_2+\S^3_1}{e-2} 
\end{align*}
where, by Lemma~\ref{GaussDie1}, $\gauss{S^3_3}{(c-1)e-2}$ and $\gauss{S^3_3}{(c-1)e-1}$ are both zero unless $\S^3_3=(c-1)e-2$ or $\S^3_3=(c-1)e-1$.  If $\S^3_3 = (c-1)e-2$ then $\S^3_1+S^3_2  = e-1+t$ and
\[b(\S) = q^{((c-1)e-1) t} \gauss{e-1+t}{t} =0\]
by Lemma~\ref{GaussDie}.  If $\S^3_3 =  (c-1)e-1$ then $\S^3_1+S^3_2  = e-2+t$. Note that $[(c-1)e-1] = -q^{(c-1)e-1}$.  Applying Lemma~\ref{GaussSum} and Lemma~\ref{GaussDie} we have
\begin{align*}
b(\S) &= q^{((c-1)e-2)(t-1)} [(c-1)e-1] \gauss{e-2+t}{e-1} - q^{((c-1)e-1)t+1} \gauss{e-2+t}{e-2} \\
&= -q^{((c-1)e-2)t+1} \left( \gauss{e-2+t}{e-1}+q^t \gauss{e-2+t}{e-2}\right) \\
& = -q^{((c-1)e-2)t+1} \gauss{e-1+t}{t} \\
&=0
\end{align*}
as required.  

Now take $d=2$ and $1 \leq t \leq e-1$ and suppose that $\S \in \mathcal{D}(2,t)$.  If $\A \in \mathcal{A}$ note that $\A^3_3 = (c-1)e-2$.  Then
\begin{align*}
b_{\mathcal{A}}(\S) & = \sum_{\A \xrightarrow{2,t} \S} (-1)^{\A^3_3-\S^3_3} q^{\binom{\A^3_3-\S^3_3}{2} + \S^3_3 t} \gauss{(b-1)e-2+t-\A^3_3}{(b-1)e-2-\S^3_3} q^{\A^4_3(\S^5_3-\A^5_4)} \gauss{\S^4_3}{\A^4_3} \gauss{\S^5_3}{\A^5_3}  \\
 &=  (-1)^{(c-1)e-2-\S^3_3} q^{\binom{(c-1)e-2-\S^3_3}{2} + \S^3_3 t} \gauss{(b-c)e+t}{(b-1)e-2-\S^3_3} \sum_{\A^4_3+\A^5_3=e-1}  q^{\A^4_3(\S^5_3-\A^5_4)} \gauss{\S^4_3}{\A^4_3} \gauss{\S^5_3}{\A^5_3} \\
&=  (-1)^{(c-1)e-2-\S^3_3} q^{\binom{(c-1)e-2-\S^3_3}{2} + \S^3_3 t} \gauss{(b-c)e+t}{(b-1)e-2-\S^3_3} \gauss{S^4_3+\S^5_3}{e-1}
\end{align*}
and the same argument shows that 
\[b_{\mathcal{B}}(\S) = (-1)^{(c-1)e-1-\S^3_3} q^{\binom{(c-1)e-1-\S^3_3}{2} + \S^3_3 t} \gauss{(b-c)e-1+t}{(b-1)e-2-\S^3_3} \gauss{S^4_3+\S^5_3}{e-2}.\]
Note that if $\A \xrightarrow{2,t} \S$ for some $\A \in \mathcal{A}$ then $e-1 \leq \S^4_3 + \S^5_3 \leq 2e-2$ and if  $\B \xrightarrow{2,t} \S$ for some $\B \in \mathcal{B}$ then $e-2 \leq \S^4_3 + \S^5_3 \leq 2e-3$.  So by Lemma~\ref{GaussDie}, $b(\S)=0$ unless $\S^4_3 + \S^5_3 = e-2$ or  $\S^4_3 + \S^5_3 = e-1$.  If  $\S^4_3 + \S^5_3 = e-2$  then
\[b(\S) = (-q) q^{\S^3_3 t} \gauss{(b-c)e-1+t}{t}=0\]
by Lemma~\ref{GaussDie}.  If $\S^4_3+\S^5_3 =e-1$ then recall that $[e-1] = -q^{e-1} = - q^{(b-c)e  -1}$.  Then
\begin{align*}
b(\S) & = q^{\S^3_3 t} \gauss{(b-c)e+t}{t} + (q) q^{\S^3_3 t} \gauss{(b-c)e-1+t}{t-1}[e-1]\\ 
&= q^{\S^3_3 t} \left( \gauss{(b-c)e+t}{t} - q^{(b-c)e} \gauss{(b-1)e+t-1}{t-1} \right)\\
&= q^{\S^3_3 t} \gauss{(b-c)e+t-1}{t} \\
&=0 \end{align*}
by Lemma~\ref{GaussSum} and Lemma~\ref{GaussDie}.  

Finally take $d=1$ and $1 \leq t \leq 2e-2$ and suppose that $\S \in \mathcal{D}(1,t)$.
By Lemma~\ref{AllTogetherNow}  
\begin{align*}
b_{\mathcal{A}}(\S) & = \sum_{\A \xrightarrow{2,t} \S} (-1)^{\A^2_2-\S^2_2} q^{\binom{\A^2_2-\S^2_2}{2}+\S^2_2 t}\gauss{(a-b+1)e-1+t}{ae-3-\S^2_2} q^{\A^3_2(\S^4_2-\A^4_2)}q^{(\A^3_2+\A^4_2)(\S^5_2-\A^5_2)} \gauss{\S^3_2}{\A^3_2}\gauss{\S^4_2}{\A^4_2}\gauss{\S^5_2}{\A^5_2} \\
&= (-1)^{\A^2_2-\S^2_2} q^{\binom{\A^2_2-\S^2_2}{2}+\S^2_2 t}\gauss{(a-b+1)e-1+t}{ae-3-\S^2_2} \sum_{\A^3_2+\A^4_2+\A^5_2=e-1} q^{\A^3_2(\S^4_2-\A^4_2)}q^{(\A^3_2+\A^4_2)(\S^5_2-\A^5_2)} \gauss{\S^3_2}{\A^3_2}\gauss{\S^4_2}{\A^4_2}\gauss{\S^5_2}{\A^5_2} \\
&= (-1)^{\A^2_2-\S^2_2} q^{\binom{\A^2_2-\S^2_2}{2}+\S^2_2 t}\gauss{(a-b+1)e-1+t}{ae-3-\S^2_2} \gauss{\S^3_2+\S^4_2+\S^5_2}{e-1}.
\end{align*}
Since $e-1 \leq \S^3_2+\S^4_2+\S^5_2 \leq 2e-2$, Lemma~\ref{GaussDie} shows that the last term is zero unless  $\S^3_2+\S^4_2+\S^5_2 = e-1$.  In this case $1 \leq t \leq e-1$ and $\S^2_2 = (b-1)e-2$ and so $b_{\mathcal{A}}(\S)$ has a factor
\[\gauss{(a-b+1)e-1+t}{t} = 0\]
by Lemma~\ref{GaussDie}.  An identical argument shows that $b_{\mathcal{B}}(\S)=0$.

This completes the proof that $\Phi(m_\mu h_{d,t})=0$ for all $1 \leq d \leq 4$ and all $1 \leq t \leq \mu_{d+1}$.      
\end{proof}

\end{document}